\documentclass[11pt]{article}
\usepackage[english]{babel}
\usepackage{epsfig}
\usepackage{amsfonts}
\usepackage{amssymb}
\usepackage{amsmath}
\usepackage{amsthm}
\usepackage{latexsym}
\usepackage{graphicx}
\usepackage{bm}
\usepackage{tabularx}
\usepackage{booktabs}
\usepackage{enumerate}
\usepackage{caption}
\usepackage{wrapfig}
\usepackage{geometry}
\usepackage{mathtools}
\usepackage[final]{microtype}
\usepackage{eqnarray}
\usepackage{enumitem}

\usepackage{graphicx}
\usepackage{subcaption}

\usepackage[title]{appendix}

\usepackage[numbers]{natbib}
\usepackage{comment}

\usepackage{bbm}
\usepackage{url}
\usepackage{hyperref}
\usepackage{color}
\usepackage{float}

\usepackage{hyperref}
\hypersetup{colorlinks=true, citecolor=blue, linkcolor=red, urlcolor=green}

\usepackage{bbm}
\usepackage{url}
\usepackage{color}
\usepackage{xcolor}
\usepackage[section]{placeins}

\graphicspath{{New_Graphs/}}

\theoremstyle{plain}
\newtheorem{theorem}{Theorem}[section]
\newtheorem{lemma}[theorem]{Lemma}
\newtheorem{corollary}[theorem]{Corollary}

\theoremstyle{definition}

\theoremstyle{remark}

\title{Erlang Loss Model with Energy Constrained Servers}
\author{ 
Elijah Kulikov  \\ Bronx High School of Science \\ {eakulikov9@gmail.com} 
\and
Eliezer Fuentes-Quezada  \\ School of Operations Research and Information Engineering \\ Cornell
University \\ {eif8@cornell.edu}
\and 
Jamol Pender \footnote{Corresponding Author}
\\
School of Operations Research and Information Engineering
\\
Cornell University
\\
{jjp274@cornell.edu}
}
\date{\today}
\usepackage{natbib}
\usepackage{graphicx}
\usepackage{amsmath}
\date{}

\begin{document}
\maketitle

\begin{abstract}
In this paper, we study an Erlang-type loss system with energy-constrained servers. Each server is equipped with a finite battery and becomes temporarily unavailable for service when its energy is depleted, entering a charging phase before returning to operation upon full recharge. Customers who arrive to find all servers unavailable are blocked and lost immediately.  We characterize the steady-state behavior of this two-dimensional Markov process by extending classical truncation results for Jackson networks to incorporate energy dynamics. This yields a closed-form product-form stationary distribution, from which we derive explicit expressions for the steady-state moments and the blocking probability.  Finally, we establish that the corresponding M/G/$k$/$k$ queue with stochastic charging is insensitive to both the service-time and charging-time distributions, depending only on their means.  Thus, we extend the insensitivity property of the Erlang loss queue to the stochastic server setting.
\end{abstract}

\section{Introduction}

Queues with stochastic server dynamics and with blocking are important models in queueing theory that capture systems where arriving customers or jobs may be denied service if the system is full. Unlike traditional queues with infinite waiting space, these models impose a strict capacity constraint, leading to the phenomenon of blocking. This makes the dynamics inherently more complex, as the system behavior depends not only on arrival and service rates but also on the interaction between the number of servers, the queue capacity, and the stochastic nature of customer arrivals and service times. These features make stochastic server queues with blocking particularly relevant in practical settings where resources are limited and overflows must be carefully managed.
The applications of blocking queues span a wide range of industries and technologies. In telecommunications, for instance, call centers and data packet networks often have limited channels or buffers, and blocking can directly impact quality of service. In healthcare, emergency departments or intensive care units face similar constraints, where the inability to admit a patient immediately may have severe consequences. Manufacturing and service operations also frequently encounter scenarios where machines or servers cannot handle additional jobs once full, requiring careful design and control to minimize lost throughput. Modeling these systems accurately is crucial to predicting performance, optimizing resource allocation, and designing robust operational policies.
A central focus in the analysis of stochastic server queues with blocking is the characterization of the steady-state distribution. Obtaining the stationary distribution allows practitioners and researchers to compute key performance metrics such as blocking probabilities, average occupancy, and system throughput. These quantities provide actionable insight into system efficiency, customer experience, and potential bottlenecks. From a theoretical perspective, the steady-state distribution also enables the validation of approximations, development of control policies, and derivation of performance bounds, which are all essential for both analysis and practical implementation.
In this paper, we focus on the rigorous computation and approximation of the steady-state distribution for a stochastic server queue with blocking. By developing methods to characterize the joint distribution of the number of busy servers and waiting customers, we provide a foundation for accurately quantifying blocking and congestion in finite-capacity systems. Our results not only contribute to the theoretical understanding of constrained stochastic service systems but also offer practical tools for engineers and decision-makers in fields ranging from telecommunications to healthcare and manufacturing. Ultimately, the ability to analyze and predict system behavior under blocking is critical for designing efficient, reliable, and high-performing service operations.

Recent work by \citet{fuentes2026queues} considers an Erlang-A queueing model and are inspired by drone delivery systems, where drones act as mobile servers that must periodically recharge between delivery tasks, see for example \citet{grippa2019drone, cokyasar2021designing, pinto2020network, huang2020drone, shavarani2019congested, seakhoa2019revenue, khamidehi2022dynamic, munishkin2023traffic, ghosh2023performance}. After completing a delivery (service), a drone cannot immediately take on a new request until it returns to a charging station and restores its energy. This creates a natural queueing framework where customer orders arrive randomly, and drones cycle between delivery and charging phases. Understanding the queue dynamics under such constraints allows system designers to balance the number of drones, charging docks, and expected wait times. Analytical and simulation-based models of queues with charging servers can help optimize fleet sizes, charging policies, and scheduling algorithms to minimize customer delay while maintaining energy efficiency.

In this paper, we consider the M/M/$k$/$k$ queue with charging servers because it provides a natural and analytically tractable framework for modeling systems with limited service capacity and no waiting space under energy constraints. In such systems, each server alternates between active service and a charging phase, and customers that arrive when all servers are occupied (either busy or charging) are blocked and lost. This captures the operational reality of many energy-dependent service environments such as fleets of autonomous machines, medical devices, and computational units where tasks cannot queue indefinitely, and service availability is restricted by both capacity and energy replenishment requirements. The M/M/$k$/$k$ structure allows us to incorporate stochastic arrivals, exponential service and charging times, and probabilistic transitions between service and charging states while maintaining analytical tractability through a continuous-time Markov chain formulation. Studying this model enables us to quantify key performance measures, such as the blocking probability and server utilization.  We show that the blocking probability can be approximated by an averaging of the Erlang-B formula with a binomial distribution. 

\subsection{Contributions of Paper}
In this paper, we make the following contributions:
\begin{itemize}
    \item Analyze a new queueing model with stochastic servers and blocking.
    \item Show the steady queue length and number of charging servers is given by a bivariate Poisson distribution truncated along a hyperplane.
    \item We finally conclude that the M/G/$k$/$k$ queueing with charging servers is insensitive to both the service-time and charging-time distributions.  
\end{itemize}

\subsection{Organization of Paper}
The remainder of the paper is organized as follows. In Section \ref{Sec_2}, we introduce the stochastic model for the M/M/$k$/$k$ queue with stochastic servers. We derive the associated functional forward equations and establish that the steady-state distribution is a bivariate Poisson distribution truncated along a hyperplane. In Section \ref{Sec_3}, we extend the analysis to the M/G/$k$/$k$ queue with stochastic servers and show, via a phase-type argument, that the system is insensitive to both the service-time and charging-time distributions. We conclude the paper in Section \ref{conclusion}.

\section{M/M/$k$/$k$ Queue with Server Charging} \label{Sec_2}

\subsection{Model Description}
In this section, we consider an M/M/$k$/$k$ queue (Erlang loss system) where servers may go to charge after completing service. In this queue, we have $k$ identical servers, each capable of serving one customer at a time or entering a charging state upon service completion. As a result, we define the two stochastic processes $Q(t)$ and $C(t)$ as
\[
Q(t) = \text{number of busy servers at time } t, \qquad 
\]
\[
C(t) = \text{number of servers currently charging at time } t.
\]
Since the number of servers is bounded, we have the following constraint on the stochastic processes $(Q(t),C(t))$
\[
Q(t) + C(t) \le k.
\]

We assume that customers arrive to the queue according to a Poisson process with rate $\lambda$, and each busy server provides service at rate $\mu$. Upon completing service, a server either immediately becomes available for a new task with probability $(1-p)$ or enters a charging phase with probability $p$. Each charging server completes its charging at rate $\gamma$, which may depend on the number of servers currently charging (for instance, due to shared charging infrastructure or power limitations).

\subsection{State Transitions}

From a given state $(Q,C)$, the possible transitions and their corresponding rates are:
\begin{enumerate}
    \item \textbf{Arrival:} if $Q+C < k$, a new customer occupies a free server,
    \[
    (Q,C) \to (Q+1,C), \quad \text{rate } \lambda.
    \]
    Otherwise, the arrival is blocked.

    \item \textbf{Service completion without charging:}
    \[
    (Q,C) \to (Q-1,C), \quad \text{rate } Q \mu (1-p).
    \]

    \item \textbf{Service completion with charging:}
    \[
    (Q,C) \to (Q-1,C+1), \quad \text{rate } Q \mu p.
    \]

    \item \textbf{Charging completion:}
    \[
    (Q,C) \to (Q,C-1), \quad \text{rate } \gamma C.
    \]
\end{enumerate}
The state space of the joint queueing and charging server processes is given by
\[
\mathcal{S} = \{ (Q,C) \mid 0 \le Q \le k, \; 0 \le C \le k, \; Q+C \le k \}.
\]

\subsection{Functional Forward Equations}

We start by deriving the functional version of the forward equations for the M/M/$k$/$k$ queue with charging servers. We then use it to derive the approximations develop later in the sequel. Let $f$ be any bounded, real-valued function on the state space of $\mathcal{S}$ . Then
\begin{eqnarray}
     \frac{d}{dt}\mathbb{E}[f(Q(t),C(t))] &=& \lambda  \mathbb{E}[ \left( f(Q(t)+1,C(t)) - f(Q(t),C(t)) \right) \cdot \textbf{1}_{Q(t) < k_{max} - C(t)}] \\
     &+& \mu (1-p) \mathbb{E}[ \left( f(Q(t)-1,C(t)) - f(Q(t),C(t)) \right) \cdot  Q(t)] \\
     &+& \mu p \mathbb{E}[ \left( f(Q(t)-1,C(t)+1) - f(Q(t),C(t)) \right) \cdot  Q(t)] \\
     &+& \gamma \mathbb{E}[ \left( f(Q(t),C(t)-1) - f(Q(t),C(t)) \right) \cdot  C(t)].
\end{eqnarray}
Moreover, the mean and covariance of (Q(t),C(t)) satisfy the following differential equations
\begin{eqnarray*}
      \frac{d}{dt}\mathbb{E}[Q(t)] &=& \lambda  \mathbb{E}[ \textbf{1}_{Q(t) < k - C(t)}] - \mu  \mathbb{E}[Q(t)]  \\
      \frac{d}{dt}\mathbb{E}[C(t)] &=&   \mu p \mathbb{E}[Q(t)]  - \gamma \mathbb{E}[ C(t) ] \\
      \frac{d}{dt}\mathrm{Var}[C(t)] &=&  \gamma \mathbb{E}[C(t)] - 2 \gamma \mathrm{Var}[C(t)] + 2 \mu p \mathrm{Cov}[Q(t),C(t)] + \mu p \mathbb{E}[Q(t) ]
      \\
    \frac{d}{dt}\mathrm{Var}[Q(t)] &=& \lambda  \mathbb{E}[ \textbf{1}_{Q(t) < k - C(t)}] +  \mu  \mathbb{E}[Q(t)] + 2 \lambda \mathrm{Cov}\left[Q(t), \textbf{1}_{Q(t) < k - C(t)} \right] - 2 \mu \mathrm{Var}[Q(t)] \\
      \frac{d}{dt}\mathrm{Cov}[Q(t),C(t)] &=& \lambda \mathrm{Cov}[C(t),\textbf{1}_{Q(t) < k - C(t)} ] - (\mu + \gamma) \mathrm{Cov}[Q(t),C(t)]  - \mu p \mathrm{Var}[Q(t)] + \mu p \mathbb{E}[Q(t)] .
\end{eqnarray*}

\subsection{Infinite Server Observation}

We begin by observing that, in the absence of the indicator function restricting arrivals, the joint process $(Q(t), C(t))$ evolves as a two-dimensional infinite-server Jackson network. Classical results by \citet{massey1993networks} show that such networks admit a product-form distribution in both transient and steady-state regimes. In addition, \citet{kelly1979reversibility, kelly1991loss} shows the Markovian infinite-server Jackson network is reversible, which allows us to invoke a truncation principle for constrained networks. This observation forms the foundation for the result stated next.

Having established that $(Q(t), C(t))$ constitutes a pair of infinite-server nodes in a Jackson network, we now characterize its steady-state behavior. The joint stationary distribution is given by a product of independent Poisson laws corresponding to each node. Imposing the linear constraint on the state space, we apply the truncation theorem for reversible Jackson networks by \citet{kelly1991loss} to obtain the steady-state distribution as a truncated version of this bivariate Poisson law. The precise form of this distribution is presented in the following theorem.

\begin{theorem}[Truncated Poisson for $Q+C\le k$]
Let $Q \sim \mathrm{Poisson}\left(\frac{\lambda}{\mu} \right)$, $C \sim \mathrm{Poisson}\left( \frac{\lambda p}{\gamma}\right)$, independent.  
Define the truncated state space
\[
\mathcal{S}_k := \{ (Q,C) \in \mathbb{Z}_{\ge 0}^2 : Q + C \le k \}.
\]

The truncated joint distribution is
\begin{align}
\mathbb{P}_{\rm trunc}(Q=x,C=y) 
= \frac{ \frac{\rho_Q^x}{x!} \frac{\rho_C^y}{y!} }{ Z_k}, \quad (x,y) \in \mathcal{S}_k,
\end{align}

where the normalization constant $Z_k$ is given by
\[
Z_k := \sum_{i=0}^{k}\sum_{j=0}^{k-i} \frac{\left(\frac{\lambda}{\mu} \right)^i}{i!} \frac{\left( \frac{\lambda p}{\gamma}\right)^j}{j!} = e^{\frac{\lambda}{\mu} + \frac{\lambda p}{\gamma}} \frac{\Gamma \left( k+1, \frac{\lambda}{\mu} + \frac{\lambda p}{\gamma} \right)}{\Gamma(k+1)}
\]
and where the incomplete gamma function is given by
\[ \Gamma \left( k, x \right)  = \int_{x}^{\infty} t^{k-1} e^{-t}  dt.\]
\begin{proof}
First rewrite the double sum using $n=i+j$:
\begin{align}
Z_k 
&= \sum_{n=0}^{k} \sum_{i=0}^{n} \frac{\rho_1^i}{i!} \frac{\rho_2^{\,n-i}}{(n-i)!}.
\end{align}

Use the identity
\[
\frac{1}{i!(n-i)!} = \frac{1}{n!} \binom{n}{i},
\]
and represent the binomial coefficient using the Beta identity
\[
\binom{n}{i} 
= \frac{1}{B(i+1,n-i+1)} 
\int_0^1 x^i (1-x)^{n-i} \, dx,
\]
where $B(a,b) = \int_0^1 x^{a-1}(1-x)^{b-1} dx$ is the Beta function.

Thus,
\begin{align}
\frac{\rho_1^i}{i!} \frac{\rho_2^{\,n-i}}{(n-i)!}
&= \frac{1}{n!} \binom{n}{i} \rho_1^i \rho_2^{\,n-i} \\
&= \frac{1}{n!} \int_0^1 \binom{n}{i} x^i (1-x)^{n-i} \rho_1^i \rho_2^{\,n-i} \, dx.
\end{align}

Summing over $i=0,\dots,n$ and interchanging sum and integral,
\begin{align}
\sum_{i=0}^{n} \frac{\rho_1^i}{i!} \frac{\rho_2^{\,n-i}}{(n-i)!}
&= \frac{1}{n!} \int_0^1 \sum_{i=0}^{n} \binom{n}{i} (\rho_1 x)^i (\rho_2 (1-x))^{n-i} \, dx.
\end{align}

The inner sum is a binomial expansion:
\[
\sum_{i=0}^{n} \binom{n}{i} (\rho_1 x)^i (\rho_2 (1-x))^{n-i}
= \big( \rho_1 x + \rho_2 (1-x) \big)^n.
\]

Hence,
\begin{align}
\sum_{i=0}^{n} \frac{\rho_1^i}{i!} \frac{\rho_2^{\,n-i}}{(n-i)!}
&= \frac{1}{n!} \int_0^1 \big( \rho_1 x + \rho_2 (1-x) \big)^n \, dx.
\end{align}

Now observe that
\[
\rho_1 x + \rho_2 (1-x) = \rho_2 + (\rho_1 - \rho_2)x,
\]
so the integral evaluates to
\[
\int_0^1 (\rho_1 x + \rho_2 (1-x))^n dx = \frac{(\rho_1+\rho_2)^n}{n+1}.
\]

Thus,
\begin{align}
\sum_{i=0}^{n} \frac{\rho_1^i}{i!} \frac{\rho_2^{\,n-i}}{(n-i)!}
= \frac{(\rho_1+\rho_2)^n}{n!}.
\end{align}

Therefore,
\begin{align}
Z_k = \sum_{n=0}^{k} \frac{(\rho_1+\rho_2)^n}{n!}.
\end{align}

Finally, using the incomplete gamma function identity for the truncated exponential sum,
\[
\sum_{n=0}^{k} \frac{x^n}{n!} = e^{x} \frac{\Gamma(k+1,x)}{\Gamma(k+1)},
\]
we obtain
\[
Z_k = e^{\rho_1+\rho_2} \frac{\Gamma(k+1,\rho_1+\rho_2)}{\Gamma(k+1)}.
\]

Substituting $\rho_1 = \frac{\lambda}{\mu}$ and $\rho_2 = \frac{\lambda p}{\gamma}$ gives the result.
\end{proof} 
\end{theorem}
With the steady-state distribution characterized, we can use it to quantify the blocking probability of our queueing system. In particular, blocking occurs precisely when the system reaches the boundary of its feasible region, that is, when the total number of customers in queue and servers in the charging state is equal to the system capacity. Thus, the blocking probability can be expressed as the sum of the mass over all states satisfying $Q + C = k$. This is obtained by summing the steady-state probabilities over this boundary set. The resulting expression provides an explicit characterization of the likelihood that an arriving customer is denied entry, and is stated formally in the result below.

\begin{theorem}[The Blocking Probability]
Let
\[
\rho_Q = \frac{\lambda}{\mu}, \qquad 
\rho_C = \frac{\lambda p}{\gamma}, \qquad 
\rho = \rho_Q + \rho_C,
\]
and define
\[
Z_k = \sum_{n=0}^k \frac{\rho^n}{n!}.
\]
Then, the blocking probability is given by
\[
\mathbb{P}(Q+C = k) = \frac{\rho^k/k!}{Z_k}.
\]
\begin{proof}
We begin by computing the marginal distribution of $N = Q + C$:
\begin{align}
\mathbb{P}(N=n) 
&= \sum_{q=0}^n \mathbb{P}(Q=q, C=n-q) \\
&= \frac{1}{Z_k} \sum_{q=0}^n \frac{\rho_Q^q}{q!} \frac{\rho_C^{n-q}}{(n-q)!}.
\end{align}

Using the binomial identity
\[
\sum_{q=0}^n \frac{\rho_Q^q}{q!} \frac{\rho_C^{n-q}}{(n-q)!}
= \frac{(\rho_Q + \rho_C)^n}{n!}
= \frac{\rho^n}{n!},
\]
we obtain
\[
\mathbb{P}(N=n) = \frac{\rho^n/n!}{Z_k}, \quad n=0,\dots,k.
\]
Thus, $N$ is a Poisson$(\rho)$ random variable conditioned on $N \le k$.  Finally, we set the total queue length and number of charging servers equal to $k$ and therefore we obtain 
\[
\mathbb{P}(Q+C = k)
= \mathbb{P}(N=k)
= \frac{\rho^k/k!}{Z_k}.
\]
\end{proof}
\end{theorem}


It is important to note that the blocking probability in this model admits a representation analogous to the classical Erlang loss formula. However, a subtle but crucial distinction arises: the effective offered load is no longer simply $\frac{\lambda}{\mu}$, as in the standard M/M/$k$/$k$ system. Instead, it takes the form $\frac{\lambda}{\mu} + \frac{\lambda p}{\gamma}$, where the first term corresponds to the usual traffic intensity from jobs in service, and the second term accounts for the additional load induced by servers that are temporarily unavailable due to charging or other operational constraints. As in the classical setting, the total number of servers is denoted by $k$.

A useful way to interpret this result is to view the combined system of queued jobs and charging servers as a single aggregate quantity. Specifically, consider the sum of the number of customers in the system and the number of servers that are currently charging. This total can be naturally approximated by a Poisson random variable with mean equal to the effective offered load described above. Truncating this Poisson distribution at the maximum allowable level, namely the total number of servers $k$, yields a blocking probability formula that mirrors the Erlang loss formula, but applied to the combined quantity of customers and charging servers rather than to the queue length alone.

This perspective provides a useful conceptual insight. Despite the additional complexity introduced by charging dynamics, the system retains a structure closely analogous to the classical Erlang loss model. The key idea is that by aggregating queue length and the number of unavailable servers into a single Poisson variable and then applying the standard truncation argument, one obtains a closed form expression for the blocking probability. This viewpoint not only streamlines the analysis but also offers a clear probabilistic interpretation: in aggregate, the system behaves like an M/M/$k$/$k$ queue with an adjusted offered load that accounts for both active and temporarily unavailable servers.

\begin{theorem}[Joint factorial moments of the truncated Poisson product distribution]
Let $Q$ and $C$ be nonnegative integer-valued random variables with joint distribution
\[
\mathbb{P}(Q=q, C=c)
=
\frac{1}{Z_k}
\frac{\rho_Q^q}{q!}\frac{\rho_C^c}{c!},
\qquad q,c \ge 0,\; q+c \le k,
\]
where the normalizing constant is
\[
Z_k
=
\sum_{n=0}^k \frac{(\rho_Q + \rho_C)^n}{n!}.
\]

Then for all integers $r,s \ge 0$, the joint factorial moments satisfy
\[
\mathbb{E}\big[(Q)_r (C)_s\big]
=
\begin{cases}
\displaystyle
\frac{\rho_Q^r \rho_C^s}{Z_k}
\sum_{m=0}^{k-(r+s)} \frac{(\rho_Q+\rho_C)^m}{m!},
& \text{if } r+s \le k, \\[1.5em]
0, & \text{if } r+s > k,
\end{cases}
\]
where $(x)_r = x(x-1)\cdots(x-r+1)$ denotes the falling factorial.

\end{theorem}

\begin{proof}
Let $N = Q + C$. Then for any $q,c \ge 0$ with $q+c=n$, we have
\[
\frac{\rho_Q^q}{q!}\frac{\rho_C^c}{c!}
=
\frac{(\rho_Q+\rho_C)^n}{n!}
\binom{n}{q}
\left(\frac{\rho_Q}{\rho_Q+\rho_C}\right)^q
\left(\frac{\rho_C}{\rho_Q+\rho_C}\right)^c.
\]
It follows that
\[
\mathbb{P}(N=n)
=
\frac{1}{Z_k}\frac{(\rho_Q+\rho_C)^n}{n!}, \qquad 0 \le n \le k,
\]
and conditionally on $N=n$,
\[
(Q,C)\mid N=n \sim \mathrm{Multinomial}\big(n; p_Q, p_C\big),
\]
where $p_Q = \frac{\rho_Q}{\rho_Q+\rho_C}$ and $p_C = \frac{\rho_C}{\rho_Q+\rho_C}$.

Thus,
\[
\mathbb{E}[(Q)_r (C)_s]
=
\mathbb{E}\big[ \mathbb{E}[(Q)_r (C)_s \mid N] \big].
\]

For a multinomial random vector,
\[
\mathbb{E}[(Q)_r (C)_s \mid N=n]
=
(n)_{r+s} p_Q^r p_C^s.
\]
Therefore,
\[
\mathbb{E}[(Q)_r (C)_s]
=
p_Q^r p_C^s \, \mathbb{E}[(N)_{r+s}].
\]

Now compute
\[
\mathbb{E}[(N)_m]
=
\frac{1}{Z_k}
\sum_{n=m}^k (n)_m \frac{(\rho_Q+\rho_C)^n}{n!}.
\]
Using $(n)_m = \frac{n!}{(n-m)!}$, we obtain
\[
\mathbb{E}[(N)_m]
=
\frac{1}{Z_k}
\sum_{n=m}^k \frac{(\rho_Q+\rho_C)^n}{(n-m)!}.
\]
Letting $j = n-m$, this becomes
\[
\mathbb{E}[(N)_m]
=
\frac{(\rho_Q+\rho_C)^m}{Z_k}
\sum_{j=0}^{k-m} \frac{(\rho_Q+\rho_C)^j}{j!}.
\]

Combining terms and using $p_Q = \frac{\rho_Q}{\rho_Q+\rho_C}$,
$p_C = \frac{\rho_C}{\rho_Q+\rho_C}$, we obtain
\[
\mathbb{E}[(Q)_r (C)_s]
=
\frac{\rho_Q^r \rho_C^s}{Z_k}
\sum_{j=0}^{k-(r+s)} \frac{(\rho_Q+\rho_C)^j}{j!},
\]
whenever $r+s \le k$.  If $r+s > k$, then $(Q)_r (C)_s = 0$ almost surely since $Q+C \le k$, and the result follows.
\end{proof}

\begin{corollary}
Let $\alpha = \frac{\rho_Q}{\rho}$, then the mean and covariance of $Q$ and $C$ are given by
\begin{align}
\mathbb{E}[Q] &= \rho_Q \frac{Z_{k-1}}{Z_k}, 
\\
\mathbb{E}[C] &= \rho_C \frac{Z_{k-1}}{Z_k},\\
\mathrm{Cov}[Q,C]
&= \alpha(1-\alpha) \left( \rho^2 \frac{Z_{k-2}}{Z_k}
- \rho^2 \left(\frac{Z_{k-1}}{Z_k}\right)^2\right).
\end{align}
\end{corollary}

\begin{theorem}[PGF via Incomplete Gamma Function]
Let $(Q,C)$ have the truncated bivariate Poisson distribution with parameters
\[
\rho_Q = \frac{\lambda}{\mu}, \quad
\rho_C = \frac{\lambda p}{\gamma}, \quad
\rho = \rho_Q + \rho_C, \quad
Z_k = \sum_{n=0}^k \frac{\rho^n}{n!}.
\]
and truncation at $k$. Then the bivariate PGF is given by
\[
G(s,t) = \mathbb{E}[ s^Q t^C ] =  \frac{\Gamma(k+1, \rho - (\rho_Q s + \rho_C t))}{\Gamma(k+1, \rho)}.
\]
\begin{proof} 
We have the following expression
\begin{align}
G(s,t) &= \frac{1}{Z_k} \sum_{n=0}^{k} \frac{\rho^n}{n!} \sum_{q=0}^{n} \binom{n}{q} \alpha^q (1-\alpha)^{n-q} s^q t^{\,n-q} \nonumber \\
&= \frac{1}{Z_k} \sum_{n=0}^{k} \frac{\rho^n}{n!} (\alpha s + (1-\alpha) t)^n \nonumber \\
&= \frac{1}{Z_k} \sum_{n=0}^{k} \frac{\big(\rho_Q s + \rho_C t \big)^n}{n!} \nonumber\\
&= \frac{1}{Z_k} \sum_{n=0}^{k} \frac{(\rho_Q s + \rho_C t)^n}{n!} \nonumber\\
&= \frac{\Gamma(k+1)}{\Gamma(k+1,\rho)} e^{-(\rho - \rho_Q s - \rho_C t)} 
\sum_{n=0}^{k} \frac{(\rho_Q s + \rho_C t)^n}{n!} e^{-(\rho_Q s + \rho_C t)} \nonumber\\
&= \frac{\Gamma(k+1)}{\Gamma(k+1,\rho)} \, e^{-(\rho - \rho_Q s - \rho_C t)} \, \frac{\Gamma(k+1, \rho - (\rho_Q s + \rho_C t))}{\Gamma(k+1)} \nonumber\\
&= \frac{\Gamma(k+1, \rho - (\rho_Q s + \rho_C t))}{\Gamma(k+1, \rho)}.
\end{align}
\end{proof}
\end{theorem}

\section{Insensitivity of the $M/G/k/k$ Loss System} \label{Sec_3}

In this section, we consider an $M/G/k/k$ queue with arrival rate $\lambda$, i.i.d.\ service times with distribution $F$ satisfying $\mathbb{E}[S] < \infty$ and i.i.d. charging times with distribution $G$ and finite mean. We also let
\[
B(F,G) := \mathbb{P}(\text{all $k$ servers are busy or charging in steady state})
\]
denote the blocking probability. It is the goal of this section, to show that the blocking probability is insenstive to the distributions $F$ and $G$ and is only a function of their means.



\begin{lemma}[Continuity of blocking probability]
Let $F_n$ be a sequence of service distributions such that
\[
F_n \Rightarrow F, \quad \mathbb{E}[S_n] \to \mathbb{E}[S].
\]
Then
\[
B(F_n) \to B(F).
\]
\begin{proof}
Let $Y^{(n)}$ and $Y$ denote the steady-state number in system for the $M/G/\infty$ systems with service distributions $F_n$ and $F$, respectively. By the standard results on Poisson driven infinite server queues \citet{eick1993physics, massey1993networks}, we have
\[
Y^{(n)} \sim \mathrm{Poisson}(\lambda \mathbb{E}[S_n]), \quad
Y \sim \mathrm{Poisson}(\lambda \mathbb{E}[S]).
\]
Since $\mathbb{E}[S_n] \to \mathbb{E}[S]$, the Poisson parameters converge, and hence
\[
Y^{(n)} \Rightarrow Y.
\]
Therefore,
\[
\mathbb{P}(Y^{(n)} \ge k) \to \mathbb{P}(Y \ge k).
\]
Using the representation $B(F_n) = \mathbb{P}(Y^{(n)} \ge k)$ and $B(F) = \mathbb{P}(Y \ge k)$ completes the proof.
\end{proof}
\end{lemma}

\begin{theorem}[Insensitivity for $M/G/k/k$ with Charging]
Consider an $M/G/k/k$ loss system with Poisson arrivals of rate $\lambda$. Each admitted customer undergoes a service time $S \sim F$ followed immediately by a charging time $ R \sim G$, where $F$ and $G$ are arbitrary nonnegative distributions with finite means $\mathbb{E}[S]$ and $\mathbb{E}[R]$.  Then the steady-state blocking probability is
\begin{align}
B &= \frac{\frac{\rho^k}{k!}}{\displaystyle \sum_{j=0}^{k} \frac{\rho^j}{j!}},
\end{align}
where
\begin{align}
\rho = \lambda \big(\mathbb{E}[S] + \mathbb{E}[R]\big),
\end{align}
and hence depends only on the means of $S$ and $R$, not on their full distributions.

\begin{proof}
Let $S$ and $R$ be nonnegative random variables with distributions $F$ and $G$ and finite means. Let $B$ denote the steady-state blocking probability in the $M/G/k/k$ loss system with i.i.d.\ service times $S$ and charging times $R$, and arrival rate $\lambda$.  We first approximate $S$ and $R$ by phase-type (PH) distributions. It is well known that PH distributions are dense in the set of probability measures on $[0,\infty)$ with respect to weak convergence, and moreover can be chosen to approximate first moments. Hence, there exist sequences $\{S^{(n)}\}$ and $\{R^{(n)}\}$ such that
\[
S^{(n)} \Rightarrow S, \qquad R^{(n)} \Rightarrow R, \qquad 
\mathbb{E}[S^{(n)}] \to \mathbb{E}[S], \qquad \mathbb{E}[R^{(n)}] \to \mathbb{E}[R],
\]
where each $S^{(n)}$ and $R^{(n)}$ is phase-type.

For each $n$, consider the $M/PH/k/k$ system with service times $S^{(n)}$ and charging times $R^{(n)}$. Since phase-type distributions are closed under convolution, the total time $T^{(n)} := S^{(n)} + R^{(n)}$ is also phase-type. Thus, the system is an $M/PH/k/k$ loss system with i.i.d.\ holding times $T^{(n)}$.  It is a classical result for $M/PH/k/k$ systems that the stationary distribution is insensitive to the specific phase-type representation and depends only on the mean holding time. In particular, the blocking probability is given by the Erlang loss formula
\[
B^{(n)} = E_k(\rho_n), \qquad 
\rho_n := \lambda \, \mathbb{E}[T^{(n)}] = \lambda \big(\mathbb{E}[S^{(n)}] + \mathbb{E}[R^{(n)}]\big),
\]
where
\[
E_k(\rho) := \frac{\frac{\rho^k}{k!}}{\displaystyle \sum_{j=0}^{k} \frac{\rho^j}{j!}}.
\]

By convergence of expectations, we have
\[
\rho_n \to \rho := \lambda \big(\mathbb{E}[S] + \mathbb{E}[R]\big).
\]
Since $E_k(\cdot)$ is continuous on $[0,\infty)$, it follows that
\[
B^{(n)} = E_k(\rho_n) \to E_k(\rho).
\]

It remains to show that $B^{(n)} \to B$. To this end, note that the $M/G/k/k$ system can be described as a Markov process on the finite state space $\{0,1,\dots,k\}$ augmented with residual service times. The corresponding sequence of processes associated with $T^{(n)}$ converges weakly to that associated with $T = S+R$ under the assumed weak convergence and convergence of first moments. 

Moreover, since the state space of the occupancy process is finite, the stationary distribution is uniquely determined and depends continuously on the underlying holding-time distribution through the associated Markov renewal kernel. This continuity follows from standard perturbation results for regenerative processes with finite state space and uniformly integrable cycle lengths, see for example \citet{asmussen2003applied, resnick2013adventures}. Consequently, the stationary blocking probabilities satisfy
\[
B^{(n)} \to B.
\]

Combining the above limits yields
\[
B = \lim_{n\to\infty} B^{(n)} = \lim_{n\to\infty} E_k(\rho_n) = E_k(\rho),
\]
that is,
\[
B = \frac{\frac{\rho^k}{k!}}{\displaystyle \sum_{j=0}^{k} \frac{\rho^j}{j!}},
\qquad 
\rho = \lambda \big(\mathbb{E}[S] + \mathbb{E}[R]\big).
\]

This establishes that the blocking probability is insensitive to the full distributions $F$ and $G$, depending only on their means.
\end{proof}
\end{theorem}

\subsection{Numerical Experiments}

In this section, we present numerical experiments demonstrating that the blocking probability is insensitive to both the service time and charging time distributions. As shown in Figures \ref{Fig:bar_graph5} - \ref{Fig:bar_graph6}, the blocking probabilities remain essentially unchanged when the mean service time and mean charging time are fixed. We compare results across five distributions for both service and charging times (exponential, gamma, lognormal, uniform, and constant). It is particularly striking that this insensitivity persists even when the server dynamics is stochastic.

\begin{figure}[htbp]
\centering
\includegraphics[scale=0.45]{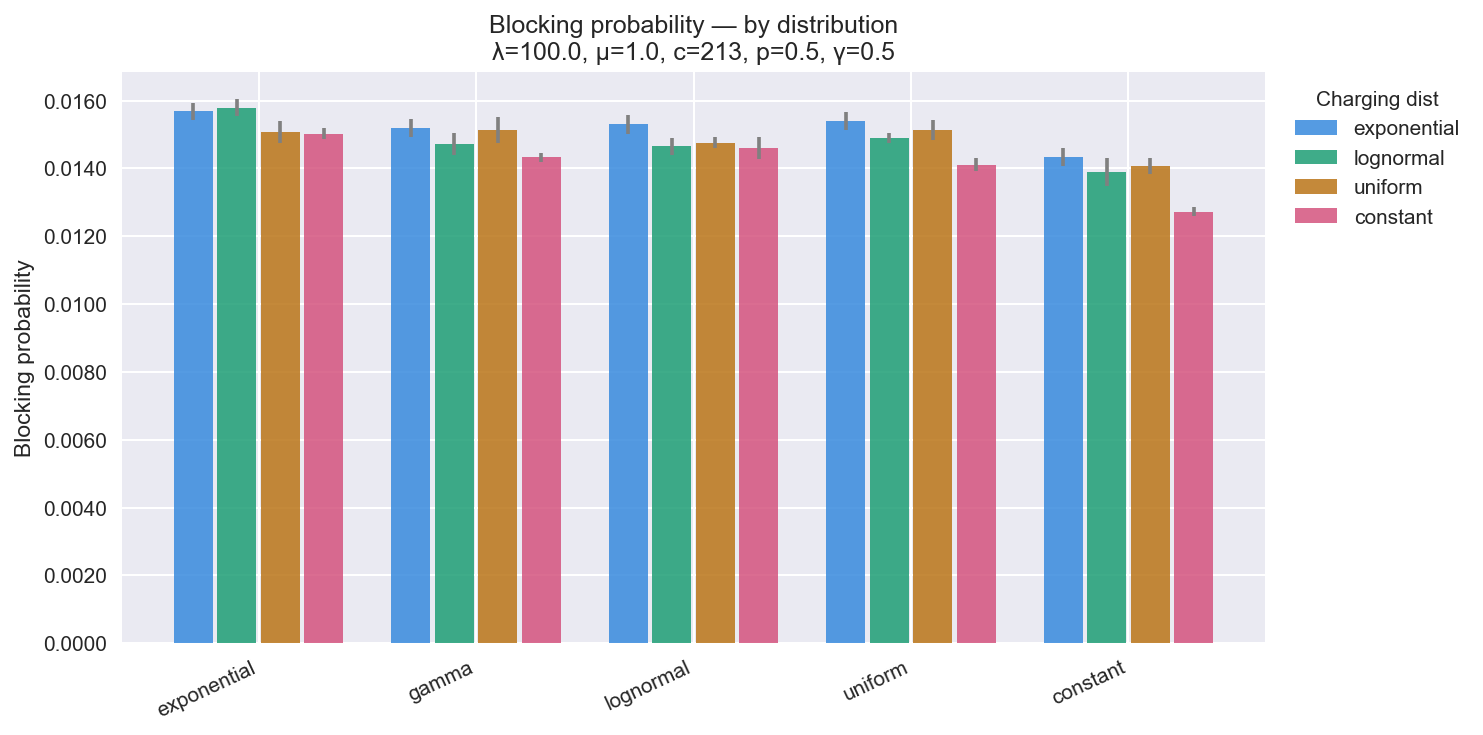}
\caption{A matrix graph of insensitivity.}  \label{Fig:bar_graph5}
\end{figure}



\begin{figure}[htbp]
\centering
\includegraphics[scale=0.45]{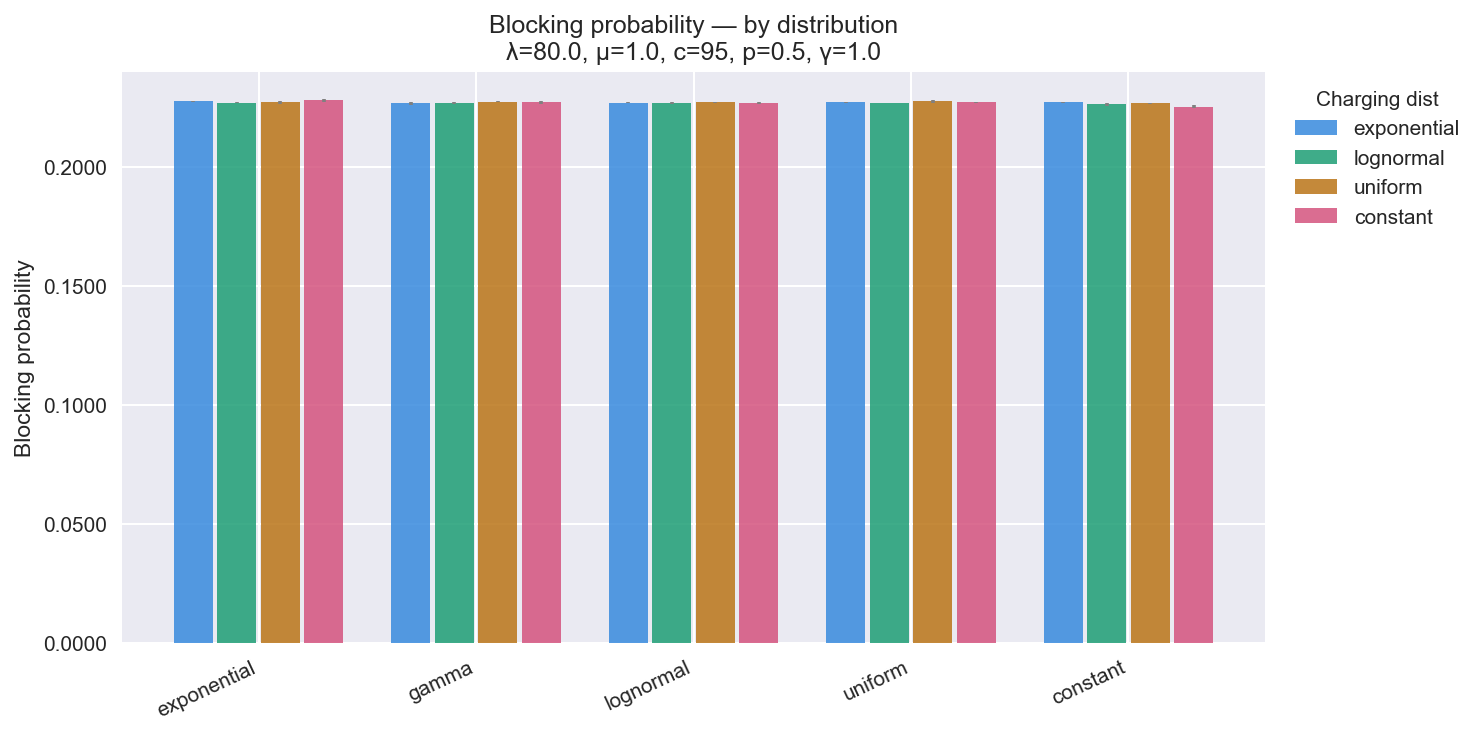}
\caption{A matrix graph of insensitivity.}  \label{Fig:bar_graph6}
\end{figure}

\section{Conclusion} \label{conclusion}
In this paper, we study the M/M/$k$/$k$ and M/G/$k$/$k$ queues with stochastic servers that intermittently leave for charging and subsequently return. We first derive an explicit expression for the steady state blocking probability and show that it can be written in a form analogous to the classical Erlang blocking formula with appropriately modified parameters. In addition, we establish that the blocking probability is insensitive to the service and charging time distributions. This is achieved by observing that the infinite server counterpart of the model with charging servers forms a Jackson network when the service and charging times are phase type. We then use the density of phase type distributions together with the continuity of the Erlang blocking formula to extend this insensitivity result to general distributions.

It is of interest to investigate whether the Hayward approximation for the G/G/$k$/$k$ queue provides an accurate approximation in the presence of general arrival processes like Hawkes processes in \citet{daw2018queues, koops2018infinite, daw2025co}. Another natural direction is to develop approximations for the blocking probability under nonstationary arrivals like in \citet{grier1997time, pender2015nonstationary, pender2020stochasticscooter}. It is conjectured that a similar approach to that used in this paper, based on aggregating the infinite server system and applying truncation, may prove effective in this setting.

Finally, an important extension is to consider networks of loss queues. Such a framework would enable the analysis of systems such as drone networks and other applications involving interconnected service nodes. It would also be relevant for modeling autonomous ride sharing systems in urban mobility networks, where vehicles alternate between serving customers and repositioning or recharging. Another promising application area is cloud computing systems with energy constrained servers, where machines cycle between processing tasks and entering low power states.

\section*{Acknowledgements}
Jamol Pender would like to acknowledge the gracious support of the American Mathematical Society via the Claytor-Gilmer Fellowship.  


\bibliographystyle{plainnat}
\bibliography{refs}

@article{fuentes2026queues,
  title={Queues with Rechargeable Servers},
  author={Fuentes-Quezada, Eliezer and Pender, Jamol},
  journal={arXiv preprint arXiv:2602.10628},
  year={2026}
}

@article{eick1993physics,
  title={{The physics of the $M_t$/G/$\infty$ queue}},
  author={Eick, Stephen G and Massey, William A and Whitt, Ward},
  journal={Operations Research},
  volume={41},
  number={4},
  pages={731--742},
  year={1993},
  publisher={INFORMS}
}

@inproceedings{khamidehi2022dynamic,
  title={Dynamic resource management for providing QOS in drone delivery systems},
  author={Khamidehi, Behzad and Raeis, Majid and Sousa, Elvino S},
  booktitle={2022 IEEE 25th International Conference on Intelligent Transportation Systems (ITSC)},
  pages={3529--3536},
  year={2022},
  organization={IEEE}
}

@article{grippa2019drone,
  title={Drone delivery systems: Job assignment and dimensioning},
  author={Grippa, Pasquale and Behrens, Doris A and Wall, Friederike and Bettstetter, Christian},
  journal={Autonomous Robots},
  volume={43},
  number={2},
  pages={261--274},
  year={2019},
  publisher={Springer}
}

@inproceedings{munishkin2023traffic,
  title={Traffic Flow Analysis for Package Delivery Drones using a Queueing Model},
  author={Munishkin, Alexey A and Pradeep, Priyank and Chour, Kenny and Kalyanam, Krishna M},
  booktitle={2023 IEEE/AIAA 42nd Digital Avionics Systems Conference (DASC)},
  pages={1--9},
  year={2023},
  organization={IEEE}
}

@inproceedings{ghosh2023performance,
  title={Performance Evaluation of Drones in FANETs using Queueing Model},
  author={Ghosh, Akash and Dash, Bibhuti Bhusan and Patra, Sudhansu Shekhar and Pandey, Trilok Nath and Pattanayak, Binod Kumar and De, Utpal Chandra},
  booktitle={2023 International Conference on Sustainable Communication Networks and Application (ICSCNA)},
  pages={43--48},
  year={2023},
  organization={IEEE}
}

@article{huang2020drone,
  title={Drone routing in a time-dependent network: Toward low-cost and large-range parcel delivery},
  author={Huang, Hailong and Savkin, Andrey V and Huang, Chao},
  journal={IEEE Transactions on Industrial Informatics},
  volume={17},
  number={2},
  pages={1526--1534},
  year={2020},
  publisher={IEEE}
}

@article{shavarani2019congested,
  title={A congested capacitated multi-level fuzzy facility location problem: An efficient drone delivery system},
  author={Shavarani, Seyed Mahdi and Mosallaeipour, Sam and Golabi, Mahmoud and {\.I}zbirak, G{\"o}khan},
  journal={Computers \& Operations Research},
  volume={108},
  pages={57--68},
  year={2019},
  publisher={Elsevier}
}

@article{cokyasar2021designing,
  title={Designing a drone delivery network with automated battery swapping machines},
  author={Cokyasar, Taner and Dong, Wenquan and Jin, Mingzhou and Verbas, {\.I}smail {\"O}mer},
  journal={Computers \& Operations Research},
  volume={129},
  pages={105177},
  year={2021},
  publisher={Elsevier}
}

@article{pinto2020network,
  title={A network design model for a meal delivery service using drones},
  author={Pinto, Roberto and Zambetti, Michela and Lagorio, Alexandra and Pirola, Fabiana},
  journal={International Journal of Logistics Research and Applications},
  volume={23},
  number={4},
  pages={354--374},
  year={2020},
  publisher={Taylor \& Francis}
}

@inproceedings{seakhoa2019revenue,
  title={Revenue-driven scheduling in drone delivery networks with time-sensitive service level agreements},
  author={Seakhoa-King, Shireen and Balaji, Paul and Alvarez, Nicolas Trama and Knottenbelt, William J},
  booktitle={Proceedings of the 12th EAI international conference on performance evaluation methodologies and tools},
  pages={183--186},
  year={2019}
}

@article{pender2020stochasticscooter,
  title={A stochastic model for electric scooter systems},
  author={Pender, Jamol and Tao, Shuang and Wikum, Anders},
  journal={arXiv preprint arXiv:2004.10727},
  year={2020}
}

@article{daw2018queues,
  title={Queues driven by Hawkes processes},
  author={Daw, Andrew and Pender, Jamol},
  journal={Stochastic Systems},
  volume={8},
  number={3},
  pages={192--229},
  year={2018},
  publisher={INFORMS}
}

@article{daw2025co,
  title={The co-production of service: Modeling services in contact centers using Hawkes processes},
  author={Daw, Andrew and Castellanos, Antonio and Yom-Tov, Galit B and Pender, Jamol and Gruendlinger, Leor},
  journal={Management Science},
  volume={71},
  number={3},
  pages={2635--2656},
  year={2025},
  publisher={INFORMS}
}

@article{koops2018infinite,
  title={Infinite-server queues with Hawkes input},
  author={Koops, David T and Saxena, Mayank and Boxma, Onno J and Mandjes, Michel},
  journal={Journal of Applied Probability},
  volume={55},
  number={3},
  pages={920--943},
  year={2018},
  publisher={Cambridge University Press}
}

@article{pender2015nonstationary,
  title={Nonstationary loss queues via cumulant moment approximations},
  author={Pender, Jamol},
  journal={Probability in the Engineering and Informational Sciences},
  volume={29},
  number={1},
  pages={27--49},
  year={2015},
  publisher={Cambridge University Press}
}

@article{massey1993networks,
  title={Networks of infinite-server queues with nonstationary Poisson input},
  author={Massey, William A and Whitt, Ward},
  journal={Queueing Systems},
  volume={13},
  number={1},
  pages={183--250},
  year={1993},
  publisher={Springer}
}

@book{kelly1979reversibility,
  title={Reversibility and stochastic networks},
  author={Kelly, Frank P},
  year={1979},
  publisher={J. Wiley}
}

@book{asmussen2003applied,
  title={Applied probability and queues},
  author={Asmussen, S{\o}ren},
  year={2003},
  publisher={Springer}
}

@book{resnick2013adventures,
  title={Adventures in stochastic processes},
  author={Resnick, Sidney I},
  year={2013},
  publisher={Springer Science \& Business Media}
}

@article{grier1997time,
  title={The time-dependent Erlang loss model with retrials},
  author={Grier, Nathaniel and Massey, William A and McKoy, Tyrone and Whitt, Ward},
  journal={Telecommunication Systems},
  volume={7},
  number={1},
  pages={253--265},
  year={1997},
  publisher={Springer}
}

@article{kelly1991loss,
  title={Loss networks},
  author={Kelly, Frank P},
  journal={The annals of applied probability},
  pages={319--378},
  year={1991},
  publisher={JSTOR}
}

\end{document}